\documentclass[11pt,reqno]{amsart}
\usepackage[foot]{amsaddr}

\usepackage{amsmath} 
\usepackage{amsfonts,amssymb,amsthm,bm}
\usepackage{amsrefs}

\usepackage{enumitem}
\usepackage[utf8]{inputenc}
\usepackage[T1]{fontenc}
\usepackage{stackrel}

\usepackage{fullpage}

\usepackage[dvipsnames]{xcolor}

\usepackage{hyperref}

\def\dd{\mathrm{\,d}} 
\newcommand\abs[1]{\left|#1\right|} 
\newcommand\norm[1]{\left\lVert#1\right\rVert} 
\newcommand\normp[2]{\left[#1\right]_{#2}} 

\DeclareMathOperator{\diam}{diam} 
\DeclareMathOperator{\dist}{\mathit{d}} 

\newcommand{\Chi}{\mathcal{X}} 
\let\emptyset\varnothing 
\DeclareMathOperator*{\essinf}{ess\,inf}
\DeclareMathOperator*{\esssup}{ess\,sup}
\DeclareMathOperator*{\loc}{loc} 
\DeclareMathOperator{\rad}{r} 

\newcommand{\R}{\mathbb{R}}

\newcommand{\BMO}{\mathrm{BMO}}

\def\Xint#1{\mathchoice
{\XXint\displaystyle\textstyle{#1}}%
{\XXint\textstyle\scriptstyle{#1}}%
{\XXint\scriptstyle\scriptscriptstyle{#1}}%
{\XXint\scriptscriptstyle\scriptscriptstyle{#1}}%
\!\int}
\def\XXint#1#2#3{\mkern3mu{\setbox0=\hbox{$#1{#2#3}{\int}$ }
\vcenter{\hbox{$#2#3$ }}\kern-.6\wd0}}

\def\dashint{\Xint-}

\theoremstyle{plain}
\newtheorem{theorem}{Theorem}
\newtheorem{propo}[theorem]{Proposition}
\newtheorem{lemma}[theorem]{Lemma}
\newtheorem{claim}[theorem]{Claim}
\numberwithin{theorem}{section}

\theoremstyle{remark}

\newtheorem{example}[theorem]{Example}

\theoremstyle{definition}
\newtheorem{dfn}[theorem]{Definition}

\begin{document}

\title{Characterizations of weak reverse H\"older inequalities on metric measure spaces}
\date{\today}

\author{Juha Kinnunen}
\address[J..K.]{Aalto University, Department of Mathematics and Systems Analysis, P.O. BOX 11100, FI-00076 Aalto, Finland}
\email{juha.k.kinnunen@aalto.fi}

\author{Emma-Karoliina Kurki}
\address[E-K.K.]{Aalto University, Department of Mathematics and Systems Analysis, P.O. BOX 11100, FI-00076 Aalto, Finland}
\email{emma-karoliina.kurki@aalto.fi}

\author{Carlos Mudarra}
\address[C.M.]{Aalto University, Department of Mathematics and Systems Analysis, P.O. BOX 11100, FI-00076 Aalto, Finland}
\email{carlos.c.mudarra@jyu.fi}



\thanks{\emph{Acknowledgements.} E.-K. Kurki has been funded by a young researcher's grant from the Emil Aaltonen Foundation.
 C. Mudarra acknowledges financial support from the Academy of Finland.}
\begin{abstract}
We present ten different characterizations of functions satisfying a weak reverse H\"older inequality on an open subset of a metric measure space with a doubling measure. Among others, we describe these functions as a class of weak $A_\infty$ weights, which is a generalization of Muckenhoupt weights that allows for nondoubling weights. Although our main results are modeled after conditions that hold true for Muckenhoupt weights, 
we also discuss two conditions for Muckenhoupt $A_\infty$ weights that fail to hold for weak $A_\infty$ weights.
\end{abstract}

\maketitle

\section{Introduction}
Assume that $(X,d,\mu)$ is a metric measure space with a doubling measure $\mu$, and let $\Omega$ be an open subset of $X$. We discuss nonnegative, locally integrable weight functions $w$ on $\Omega$ that satisfy a weak reverse H\"older inequality, namely that there exist $p>1$ and a constant $C$ such that 
\begin{equation}\label{wrevhint}
\left(\dashint_B w^p\dd\mu\right)^\frac{1}{p} \leq C \dashint_{2B}w\dd\mu
\end{equation}
for every ball  $B$ with $2B\Subset \Omega$. 
This inequality is weaker than the corresponding reverse H\"older inequality with the same ball on both sides, that is,
\begin{equation}\label{revhint}
\left(\dashint_B w^p\dd\mu\right)^\frac{1}{p} \leq C \dashint_{B}w\dd\mu
\end{equation}
for every ball $B\Subset \Omega$.
The larger ball on the right-hand side of \eqref{wrevhint} is a relaxation of \eqref{revhint} that allows for nondoubling weights. 
In addition, this kind of weights may vanish on a set of positive measure. Weak reverse H\"older inequalities with applications in the calculus of variations and partial differential equations have been studied by Meyers and Elcrat \cite{MR417568} and Giaquinta and Modica \cite{MR549962}. Functions satisfying a weak reverse H\"older inequality have also been discussed in \cites{MR3606546,MR1308005,MR3310925,MR1274087,spadaro,MR570689,MR2173373}. 

The reverse H\"older inequality \eqref{revhint} plays a central role in the theory of Muckenhoupt weights. Comprehensive lists of characterizations of Muckenhoupt $A_\infty$ weights can be found in \cites{duo_paseky,MR3473651,MR2815740, shukla_thesis}, as well as in many classic reference works.
In particular, a weight belongs to a Muckenhoupt $A_p$ class with $1<p<\infty$ (and thus to the class $A_\infty$) if and only if it satisfies the reverse H\"older inequality \eqref{revhint}. The connection is very well understood in the Euclidean case; however, some unexpected phenomena occur in more general metric measure spaces.
Str\"{o}mberg and Torchinsky \cite{MR1011673} have discussed the equivalence between various conditions for the Muckenhoupt class $A_\infty$ on spaces of homogeneous type, under the relatively strong assumption that the measure of a ball depends continuously on its radius. Kinnunen and Shukla \cite{MR3265363} discuss this question under a weaker condition called the annular decay property. This covers a relatively large class of spaces including, for instance, all length spaces. On metric measure spaces, weak reverse H\"older inequalities seem to appear even if one begins with the reverse H\"older inequality \eqref{revhint}; see \cites{MR3606546,HytonenPerezRela2012}. 
Since the weight is not necessarily doubling, it is not clear how to return to the original class of functions. The principal references to the theory of Muckenhoupt weights on metric measure spaces are \cites{MR0499948,MR1791462,MR1011673}. 

In the present article, we reconsider functions satisfying the weak reverse H\"older inequality \eqref{wrevhint}. The main theorem (Theorem \ref{thm:char}) presents no less than ten alternative characterizations of functions satisfying this condition. The reader of Section \ref{sec:char} will find that many of our characterizations mirror those of Muckenhoupt $A_\infty$ weights with the necessary modifications. These conditions include a weak $A_\infty$ condition introduced by Spadaro \cite{spadaro}, who established the equivalence with the weak reverse H\"older inequality in the Euclidean case with the Lebesgue measure. The weak $A_\infty$ class consists of locally integrable, almost everywhere nonnegative functions that satisfy the following condition:  if for every $\eta>0$ there exists $\varepsilon >0$ such that if $B$ is a ball with $2B \Subset \Omega$ and $F \subset B$ is a measurable set, then
\begin{equation}\label{nondo}
\mu(F) \leq \varepsilon \mu(B) \quad \implies \quad \int_F w \dd \mu
\leq \eta \int_{2B} w \dd \mu.
\end{equation}

Several characterizations that we present are new, while others collect and generalize existing results. In particular, we extend the result of Spadaro \cite{spadaro} to metric measure spaces with a doubling measure. Compared to the Euclidean proof, references to the Besicovitch covering theorem and continuity of measure have been replaced with suitable arguments. A recent contribution on the topic of weak $A_\infty$ weights is Anderson, Hyt\"onen, and Tapiola \cite{MR3606546}, where they show that the so-called weak Fujii--Wilson condition (our Theorem \ref{thm:char}~\ref{char-fujii}) implies a weak reverse H\"older inequality by the means of Christ-type dyadic structures. While their choice of method is justified as a deliberate departure from \cite{HytonenPerezRela2012}, we give an elementary proof of the Fujii--Wilson characterization in Proposition \ref{charc} below. In contrast to \cites{MR3606546, HytonenPerezRela2012} and the habitual choice in harmonic analysis, we do not require the weight to be defined in the entire space, but an open subset is enough. Our choice is relevant in applications to partial differential equations.

Although the assertions in Theorem \ref{thm:char} are modeled after conditions that hold true for Muckenhoupt weights, 
we also discuss two conditions for Muckenhoupt $A_\infty$ weights that fail to hold for functions that merely satisfy a weak reverse H\"older inequality. Section \ref{weakRHI} establishes the equivalence of \eqref{nondo} with the other characterizing conditions, completing the proof of the main theorem. Finally, we remark that the factor $2$ in \eqref{nondo} might as well be replaced with any other $\sigma>1$, resulting in an identical class of functions. For this question, see also \cite{MR802488}*{Theorem 2}.

\section{Preliminaries}

Throughout, we consider a complete metric measure space $\left(X,\dist,\mu\right)$, where the Borel regular measure $\mu$ satisfies the doubling condition: there exists a constant $C_d=C_d(\mu) > 1$ such that 
$$
0<\mu\left(2 B\right) \leq C_{d} \mu\left( B \right) < \infty 
$$
for all balls $B\subset X$. The doubling constant $C_d$ is assumed to be strictly larger than $1$ to exclude trivial spaces. 
The doubling condition implies that $X$ is separable, and since $X$ is complete, it is proper \cite{MR2867756}*{Proposition 3.1}. An open ball with center $x\in X$ and radius $r>0$ is denoted $B = B(x, r)$, where the center and radius are left out when not relevant to the discussion. We denote $\rad(B) = r$ when $B=B(x,r)$, and $aB = B(x, ar)$ for the ball dilated by a constant $a>0$. Various constants are denoted by the letter $C$, whose dependence on parameters is indicated in parentheses. 
For an open set $\Omega$ and an arbitrary subset $E\subset\Omega$,  the relation $E \Subset\Omega$ indicates that the closure of $E$ is a compact subset of $\Omega$.

The following lemma is related to coverings with balls, and will be of use throughout the discussion. 
\begin{lemma}\label{claimsigmacoveringballs}
Let $(X,d,\mu)$ be a metric measure space with a doubling measure $\mu$, $B \subset X$ a ball, and $\sigma>0$. There exist $N=N(\sigma, C_d)$ balls $\lbrace B_i \rbrace_i$ centered at points $x_i\in B$ such that $B \subset \bigcup_{i=1}^N B_i$ and $\rad(B_i)= \sigma \rad(B)$. 
\end{lemma}
\begin{proof} 
Let $r$ denote the radius of the ball $B$ from the assumption, and $\lbrace B \left( x, \frac{\sigma}{5} r \right) \rbrace_{x\in B}$ be an arbitrary collection of balls. By the Vitali covering lemma (\cite{MR2867756}*{Lemma 1.7}) we find a disjoint subcollection $\lbrace B \left( x_i, \frac{\sigma}{5} r \right) \rbrace_i$ such that $B \subset \bigcup_i B_i$, where $B_i = B\left( x_i, \sigma  r \right).$ Because each $B_i$ is centered at a point $x_i\in B$ it is clear that $B \subset B(x_i,  2 r)$, and the doubling condition implies
\begin{equation}\label{ztemp2}
\mu\left( B \left( x_i, \tfrac{\sigma}{5} r \right) \right) \geq C(\sigma, C_d) \mu \left( B(x_i,  2 r) \right) \geq C(\sigma, C_d) \mu (B). 
\end{equation}
Since every ball $B \left( x_i, \frac{\sigma}{5} r \right)$ is contained in $\left(1+ \frac{\sigma}{5} \right) B$, we further have 
\begin{equation}\label{ztemp3}
\sum_i \mu \left( B \left( x_i, \tfrac{\sigma}{5} r \right) \right) 
= \mu \Bigl( \bigcup_i B \left( x_i, \tfrac{\sigma}{5} r \right) \Bigr) \leq  \mu \left( \left(1+ \tfrac{\sigma}{5} \right) B \right) \leq \widetilde{C}(\sigma, C_d) \mu(B).
\end{equation}
Combining inequalities \eqref{ztemp2} and \eqref{ztemp3} implies that the cardinal of $\lbrace B_i \rbrace_i$ is finite and bounded by a number depending only on $\sigma$ and $C_d$. 
\end{proof}

Whenever $E\subset X$ is a measurable subset and the function $f$ is integrable on every compact subset of $E$ we say that $f$ is locally integrable on $E$, denoted $f\in L^1_{\loc}(E)$. If the measure $\nu$ is absolutely continuous with respect to $\mu$ and if there exists a nonnegative locally integrable function $w$ such that $d\nu=w\dd\mu$, we call $\nu$ a weighted measure with respect to $\mu$, and $w$ a weight \cite{MR1011673}. For any measurable subset $F\subset E$ and weight $w$ on $E$, we write $w(F)=\int_F w \dd \mu$. 
A weight is said to be doubling, if the measure induced by the weight is doubling.
The integral average of a function $f\in L^1(E)$ over a measurable set $F\subset E$, with $0<\mu(F)<\infty$, is abbreviated 
$$
f_F = \dashint_F f\dd\mu = \frac{1}{\mu(F)}\int_F f\dd\mu.
$$

The Hardy-Littlewood maximal function of $f\in L^1_{\loc}(X)$ is defined by
\begin{equation*}
Mf(x)=\sup_{B\ni x}\frac{1}{\mu(B)}\int_B\abs{f}\dd\mu,
\end{equation*}
where the supremum is taken over all balls $B\subset X$ containing the point $x$. We will be invoking the following weak type $(1,1)$ inequality for the maximal function, which we state without proof; see \cite{MR1011673}*{Theorem I.3}.
\begin{propo}\label{thm:weaktype}Let $(X,d,\mu)$ be a metric measure space with a doubling measure $\mu$. Then for every $f\in L^1(X)$ and $t > 0$ it holds that
$$
\mu\left(\left\{Mf>t\right\}\right) \leq \frac{C}{t}\int_X \abs{f} \dd \mu
$$
with a constant $C = C(C_d)$.
\end{propo}

Let $\Omega\subset X$ be a nonempty open subset and $1<p<\infty$. A nonnegative function $w\in L^1_{\loc}(\Omega)$ belongs to 
the Muckenhoupt class $A_p(\Omega)$ if
\begin{equation}\label{apcondition}
\normp{w}{p} = 
\sup_{\substack{B\subset \Omega}}\left(\dashint_{B}w\dd\mu\right)\left(\dashint_{B}w^{-\frac{1}{p-1}}\dd \mu\right)^{p-1}<\infty,
\end{equation}
where the supremum is taken over all balls $B\subset \Omega$. For $p=1$, we require that there exists a constant $\normp{w}{1}<\infty$ such that
\begin{equation*}
\dashint_{B}w\dd\mu \leq \normp{w}{1}\essinf_{x\in B} w(x)
\end{equation*}
for all balls $B\subset \Omega$.
In both cases, the constant denoted $\normp{w}{p}$ is called the $A_p$ constant of $w$. As for the class $A_\infty(\Omega)$, we choose the following definition (see, for instance, \cite{MR1011673}): there exist constants $0< \varepsilon, \eta<1$ such that for any ball $B\subset \Omega$ and measurable subset $A\subset B$, 
$$
\mu(A)\leq \varepsilon \mu(B) \quad\implies\quad w(A) \leq\eta w(B). 
$$

Muckenhoupt $A_p$ weights with $1<p<\infty$ are those weights $w$ for which the maximal operator is bounded on the weighted space $L^p(X; w)$.
Moreover, Muckenhoupt classes are closely connected to the functions of bounded mean oscillation ($\BMO$). Let $\Omega\subset X$ be a nonempty open set.
We say that a function $f\in L^1_{\loc}(\Omega)$ belongs to $\BMO(\Omega)$ if
$$
\norm{f}_{\BMO(\Omega)} = \sup_{B \subset \Omega} \dashint_B \abs{f-f_B}\dd\mu <\infty. 
$$
The constant $\norm{f}_{\BMO(\Omega)}$, strictly speaking a seminorm, is spoken of as the $\BMO$ norm of $f$. Whenever $p>1$, we have $\BMO(\Omega) = \left\{\alpha \log w\mathbin{:} \alpha\geq 0 \text{ and }w\in A_p(\Omega)\right\}$; for details see \cite{shukla_thesis}*{Theorem 3.11} and the remark thereafter. 

We say that a nonnegative $w\in L^1_{\loc}(\Omega)$ belongs to the reverse H\"older class $RH_p(\Omega)$, with $1<p<\infty$, if there exists a constant $C$ such that 
\begin{equation}\label{strongrevp}
\left(\dashint_B w^p\dd\mu\right)^\frac{1}{p} \leq C \dashint_{B}w\dd\mu
\end{equation}
for every ball $B\Subset \Omega$.
Muckenhoupt $A_\infty$ weights are precisely the weights with this property in $\mathbb{R}^n$ and, more generally, in spaces with the annular decay property; see, for instance, \cite{MR807149} and \cite{shukla_thesis}.
In particular, a weight in $RH_p(\Omega)$ is doubling over balls $B$ with $2B\Subset\Omega$.
The weak reverse H\"older class $WRH_p(\Omega)$ is a nondoubling analogy of $RH_p(\Omega)$. We say that an almost everywhere nonnegative function $w\in L^1_{\loc}(\Omega)$ belongs to $WRH_p(\Omega)$ if there exist $1<p<\infty$ and a constant $C$ such that
\begin{equation}\label{revp}
\left(\dashint_B w^p\dd\mu\right)^\frac{1}{p} \leq C \dashint_{2B}w\dd\mu
\end{equation}
for every ball $B$ with $2B\Subset \Omega$.  Note that for such balls $B,$ we have $w\in L^1(2B)$.

We also consider the limiting case $p=\infty$ of the weak reverse H\"older condition.
\begin{dfn} For an open set $\Omega\subset X$ we say that the weight $w$ belongs to the weak $RH_\infty$ class, denoted $WRH_\infty(\Omega)$, if there exists a constant $C>0$ such that
\begin{equation}\label{definitionWRHinfty}
\esssup_B w \leq C \dashint_{2B} w \dd \mu
\end{equation}
for every ball $B$ with $2B\Subset \Omega$.
\end{dfn}

\begin{example}\label{example:revclass}
Let $w(x)=e^x$ on $\R$ with the Lebesgue measure. Then $w\in WRH_\infty(\R)$. 
To conclude this, let $B=(a-r, a+r)$ be an interval in $\R$. Then
$$
\esssup_B w = e^{a+r},
\quad\text{and}\quad 
\dashint_{2B} w\dd x = \frac1{4r}\left(e^{a+2r}-e^{a-2r}\right).
$$
It follows that
$$
\frac{\esssup_B w}{\dashint_{2B} w\dd x } = \frac{4r e^{a+r}}{e^{a+2r}-e^{a-2r}} = \frac{4re^{3r}}{e^{4r}-1},
$$
which is a bounded function of $r>0$. On the other hand, $w\notin RH_p(\R)$ for any $1<p<\infty$, as the inequality \eqref{strongrevp} fails for intervals $B=(-r,r)$ with sufficiently large $r$.
\end{example}

From the above definitions it is immediate that a weight in $WRH_\infty(\Omega)$ belongs to $WRH_p(\Omega)$ for every $1<p<\infty$, so we have 
$$
WRH_\infty(\Omega) \subset \bigcap_{p>1} WRH_p(\Omega).
$$ 
We remark that the inclusion is proper. For instance, the weight $w(x)= \max \lbrace \log (\abs{x}^{-1}), 1 \rbrace$ on $\R$ does not belong to $WRH_\infty(\R)$ because $\esssup_{(-r,r)} w= \infty$ for $r < e^{-1}$. However, $w$ belongs to the Muckenhoupt class $A_1 $ and to $RH_p$ for every $p>1$; see \cite{MR1308005} and \cite{MR3243734}*{p.~507}. 
Furthermore, it is worth noticing that the $WRH_p(\Omega)$ classes are self-improving. Given $w\in WRH_p(\Omega)$, there exists $q>p$ such that $w\in WRH_q(\Omega)$ \cite{MR2867756}*{Theorem 3.22}.

The following lemma presents a few alternative ways to characterize the class $WRH_\infty(\Omega)$. Assertion \ref{wrhia} is precisely Theorem \ref{thm:char}~\ref{char-d} with $\alpha=1$. Assertions \ref{wrhib} and \ref{wrhic} can be compared to Theorem \ref{thm:char}~\ref{char-fujii} and \ref{char-hmm}, respectively, while being strictly stronger.
\begin{lemma}\label{lem:wrhi}
Let $\Omega\subset X$ be an open set.
\begin{enumerate}[label=\normalfont{(\roman*)}]
\item \label{wrhia} $w\in WRH_\infty(\Omega)$ is equivalent to $w_F \leq C w_{2B}$ for every measurable set $F \subset B$ with $2B \Subset \Omega$. 
\item \label{wrhib} $w\in WRH_\infty(\Omega)$ is equivalent to $M\left( w \Chi_B \right) \leq C w_{2B}$ on $B$ for every ball $B$ with $2B \Subset \Omega$.
\item \label{wrhic} $w \in WRH_\infty(\Omega)$ if and only if there exist a nondecreasing function $\phi: (1,\infty) \to (0,\infty)$ and a constant $C>0$ such that 
$$
\esssup_{B \cap \lbrace w >w_{2B} \rbrace} w \phi\left( \frac{w}{w_{2B}} \right) \leq C w_{2B}
$$
for every ball $B$ such that $2B\Subset \Omega$. In the statement, ``there exist a $\phi$\ldots and a constant $C>0$'' can be replaced with ``for all $\phi$\ldots there exists a constant $C>0$ such that\ldots''
\end{enumerate}
In the case where $w(2B)=0,$ the inequalities are trivially true for $B$.
\end{lemma}
\begin{proof} (i) Assume first that $w\in WRH_\infty(\Omega)$ and let $C>0$ be the constant in \eqref{definitionWRHinfty}. Then
$$
w(F)=\int_B w \Chi_F \dd\mu \leq \esssup_B w \int_B \Chi_F \dd\mu = \mu(F) \esssup_B w \leq C\mu(F)w_{2B},
$$
which is what we want. Conversely, assume that $w_F \leq C w_{2B}$ for every $F \subset B$. 
Let $F_t=B\cap\lbrace w>t \rbrace$ with $t>0$. By assumption
$$
\mu(F_t) t \leq w(F_t) \leq C w_{2B} \mu(F_t).
$$
This means that $\mu(F_t)>0$ implies $C w_{2B} \geq t$. Therefore, we have
\begin{align*}
\int_B w^r \dd\mu &= r \int_0^\infty t^{r-1} \mu(F_t) \dd t=r \int_0^{C w_{2B}} t^{r-1} \mu(F_t) \dd t \\
&\leq \mu(B) r \int_0^{C w_{2B}} t^{r-1} \dd t = \mu(B) \left( C w_{2B} \right)^r
\end{align*}
for every $r>1$. We thus have 
$$\left( \dashint_B w^r\dd\mu \right)^{1/r} \leq C w_{2B}
$$ 
for every $r>1$, implying \eqref{definitionWRHinfty}. 

(ii) If \eqref{definitionWRHinfty} holds, we have 
$$
M( w \Chi_B) \leq \esssup_X \left( w \Chi_B \right) = \esssup_B w \leq C w_{2B}.
$$ 
Conversely, if the condition on the right-hand side of the claim is satisfied, by applying the fact that $M( w \Chi_B ) \geq w \Chi_B$ almost everywhere in $X$, we obtain
$$
w \leq M( w \Chi_B ) \leq C w_{2B}
$$
almost everywhere in $B$. 

(iii) Let $w\in WRH_\infty(\Omega)$, and let $C$ be the constant in \eqref{definitionWRHinfty}. If $C \leq 1$, we have $ w \leq C w_{2B} \leq w_{2B}$ almost everywhere in $B$ and thus $\mu \left( B \cap \lbrace w> w_{2B} \rbrace \right)=0$, meaning that the desired estimate trivially holds. If $C >1$, then \eqref{definitionWRHinfty} gives
$$
\esssup_{B \cap \lbrace w >w_{2B} \rbrace} w \phi\left( \frac{w}{w_{2B}} \right) \leq C \phi(C) w_{2B}
$$
for every $\phi$ as in the claim.

Conversely, let $\phi$ and $C$ be as in the statement of the lemma. Then 
$$
\esssup_{B \cap \lbrace w > 2 w_{2B} \rbrace} w  
= \frac{1}{\phi(2)} \esssup_{B \cap \lbrace w >  2w_{2B} \rbrace} w \phi(2) \leq \frac{1}{\phi(2)} \esssup_{B \cap \lbrace w >  2 w_{2B} \rbrace} w \phi \left( \frac{w}{w_{2B}} \right)  \leq \frac{C }{\phi(2)}w_{2B}. 
$$
Combining this with the trivial estimate $\esssup_{B \cap \lbrace w \leq 2 w_{2B} \rbrace} w \leq 2 w_{2B}$, we end up with 
$$\esssup_{B } w \leq \left( \frac{C }{\phi(2)} +2 \right) w_{2B}.$$
\end{proof}

\section{Characterizations of weak reverse H\"older inequalities}\label{sec:char}

We are ready to take on the main theorem that characterizes the weak $A_\infty$ class in various ways. In the rest of the paper, we denote $\log^+ =\max\left\{0, \log \right\}$.

\begin{theorem}\label{thm:char}
Let $(X,d,\mu)$ be a metric measure space with a doubling measure $\mu$. Assume that $\Omega\subset X$ is an open set and that $w$ is a weight on $\Omega$. Then the following assertions are equivalent.

\begin{enumerate}[label=\normalfont{(\alph*)}]

\item\label{char-a} {\normalfont(Weak $A_\infty$)} For every $\eta>0$ there exists $\varepsilon >0$ such that if $B$ is a ball with $2B \Subset \Omega$ and $F \subset B$ is a measurable set, 
then $\mu(F) \leq \varepsilon \mu(B)$ implies that $w(F) \leq \eta w(2B)$.

\item\label{char-b} There exist $\eta, \varepsilon>0$ with $\eta < C_d^{-5}$ such that if $B$ is a ball with $2B\Subset \Omega$ and $F \subset B$ is a measurable set, 
then $\mu(F) \leq \varepsilon \mu(B)$  implies that $w(F) \leq \eta w(2B)$.

\item\label{char-c} {\normalfont(Weak reverse H\"older inequality)} There exist $p>1$ and a constant $C>0$ such that for every ball $B$ with $2B\Subset \Omega$,  it holds that 
$$
\dashint_B w^p \dd \mu \leq C \left(  \dashint_{2B} w \dd \mu \right)^p.
$$

\item\label{char-d} {\normalfont(Quantitative nondoubling $A_\infty$)} There exist constants $C, \alpha >0$ such that for every ball $B$ with $2B\Subset \Omega$ and every measurable set $F \subset B$, it holds that
$$
w(F) \leq C \left( \frac{\mu(F)}{\mu(B)} \right)^\alpha w(2B) .
$$

\item\label{char-fujii} {\normalfont(Fujii-Wilson condition)} There exists a constant $C>0$ such that
$$
\int_B M(w \Chi_B ) \dd \mu \leq C w(2B)
$$
for every ball $B$ with $2B \Subset \Omega$.

\item\label{char-dist} There exist constants $\alpha, \beta>0$ with $\beta < C_d^{-5}$ such that 
$$
w \left( B\cap\left\lbrace  w\geq  \alpha w_{2B} \right\rbrace \right) \leq \beta w(2B)
$$
 for every ball $B$ with $2B\Subset \Omega$.
 
\item\label{char-log} There exists a constant $C>0$ such that
$$
\int_B w \log^+\left(\frac{w}{w_{2B}} \right) \dd \mu \leq C w(2B)
$$
for every ball $B$ with $2B\Subset \Omega$.

\item\label{char-saw1} There exists a nondecreasing function $\phi:(0,\infty)\to(0,\infty)$ with $\phi(0^+)=0$ such that for every ball $B \subset X$ with $2 B\Subset \Omega$ and every measurable set $F \subset B$ it holds that
$$
w(F) \leq \phi \left( \frac{\mu(F)}{\mu(B)} \right) w(2B).
$$

\item\label{char-saw2} There exists a constant $C>0$ such that for every ball $B$ with $11B \Subset \Omega$ and every function $f\in \BMO(\Omega)$ with $\norm{f}_{\BMO(\Omega)} \leq 1$, it holds that
$$
\int_B\abs{ f-f_B } w \dd \mu \leq C w(2B).
$$

\item\label{char-hmm} There exist a constant $C>0$ and a nondecreasing function $\phi: (1,\infty) \to (0,\infty)$ with $\phi(\infty) = \infty$ such that 
$$
\int_{B\cap\lbrace w>w_{2B} \rbrace} w \, \phi \left(\frac{w}{w_{2B}} \right) \dd \mu \leq C w(2B)
$$
for every ball $B$ with $2B\Subset \Omega$.
\end{enumerate}
In the case where $w(2B)$ equals zero, the inequalities in \ref{char-a} -- \ref{char-hmm} are trivially satisfied for $B$.
\end{theorem}

In Section \ref{weakRHI} we show that the assertions \ref{char-a} -- \ref{char-c} are equivalent. Once this has been shown, it follows from Proposition \ref{prop-easy} that \ref{char-d} is also equivalent to these three. The remainder of this section is dedicated to showing that the assertions \ref{char-fujii} -- \ref{char-hmm} are equivalent to \ref{char-a} -- \ref{char-d}. It is most convenient to start by assuming either \ref{char-c} or \ref{char-d}, and end up with one of the qualitative claims \ref{char-a} or \ref{char-b}. 
Assertions \ref{char-saw1} and \ref{char-saw2} appear in a Euclidean setting in the early article \cite{MR654182}, while \ref{char-log} is inspired by \cite{duo_paseky}*{Theorem 5.1} for $A_\infty$ weights. 


\begin{propo}\label{prop-easy}
\ref{char-c} $\implies$ \ref{char-d} $\implies$ \ref{char-a}
\end{propo}
\begin{proof}
To conclude the first implication, we apply H\"older's inequality and \ref{char-c} to obtain
\begin{align*}
w(F) &= \int_B \Chi_Fw\dd\mu \leq \left(\int_B\Chi_F\dd\mu\right)^{1-\frac{1}{p}}\left(\int_B w^p \dd \mu \right)^\frac{1}{p}
\leq C\mu(F)\left(\frac{\mu(B)}{\mu(F)}\right)^\frac{1}{p}\dashint_{2B}w\dd\mu\\
&= C\mu(F)\left(\frac{\mu(B)}{\mu(F)}\right)^\frac{1}{p}\frac{w(2B)}{\mu(2B)}
\leq Cw(2B)\left(\frac{\mu(F)}{\mu(B)}\right)^{1-\frac{1}{p}}.
\end{align*}
The second implication is immediate by choosing $\varepsilon = \left(\eta C^{-1}\right)^{1/\alpha}$, where $C,\alpha$ are the constants in \ref{char-d}. 
\end{proof}

Anderson, Hyt\"onen, and Tapiola \cite{MR3606546} show that \ref{char-fujii} $\implies$ \ref{char-c} using weak weights defined on Christ-type dyadic systems of cubes. We prove that \ref{char-fujii} $\implies$ \ref{char-a} by elementary arguments; Section \ref{weakRHI} will complete the proof of the equivalence. 
\begin{propo}\label{charc}
\ref{char-c} $\implies$ \ref{char-fujii} $\implies$ \ref{char-a}
\end{propo}
\begin{proof}
The implication \ref{char-c} $\implies$ \ref{char-fujii} is a consequence of the $(p,p)$-strong type inequality for the maximal function. Let $p>1$ as in \ref{char-c}, and $B $ a ball with $2B \Subset \Omega$. Then
\begin{align*}
\int_B M( w \Chi_B) \dd \mu 
&=  \int_B \left( M( w \Chi_B)^p \right)^\frac{1}{p} \dd \mu 
\leq \mu(B)^{\frac{p-1}{p}} \left( \int_X M\left( w \Chi_B \right)^p \dd \mu \right)^{\frac{1}{p}} \\
&\leq \mu(B)^{\frac{p-1}{p}} \left( C(p,C_d) \int_X  \left( w \Chi_B \right)^p \dd \mu \right)^{\frac{1}{p}} 
=C(p,C_d)^\frac{1}{p}\mu(B)\left( \dashint_B w^p \dd \mu \right)^{\frac{1}{p}} \\
&\leq C(p,C_d)^\frac{1}{p} C\mu(B)\dashint_{2B} w \dd \mu \leq C(p,C_d)^\frac{1}{p} C  w(2B). 
\end{align*}
The proof of \ref{char-fujii} $\implies$ \ref{char-a} is more involved, and broken into several lemmas for the reader's convenience. We start with the following reverse $(1,1)$-weak type estimate for the maximal function. 
\begin{lemma}\label{reverseweaktypeestimate}
Let $f\in L^1(X)$ be a nonnegative function, $B\subset X$ a ball, and $\lambda >  \dashint_B f \dd\mu$. There exists a constant $C_0=C_0(C_d)>1$ such that
$$
\int_{B\cap \lbrace Mf >\lambda \rbrace} f \dd\mu \leq C_0 \lambda \mu \left( \lbrace Mf > \lambda \rbrace \right).
$$
\end{lemma}
\begin{proof}[Proof of Lemma \ref{reverseweaktypeestimate}]
In the case $\lbrace Mf \leq \lambda \rbrace = \emptyset$, the left-hand side equals $\int_B f \dd\mu $ and the right-hand side equals $C_0 \lambda \mu(X)$, and the estimate trivially holds.

Then assume that $\lbrace Mf \leq \lambda \rbrace \neq \emptyset$. 
Denote $E_\lambda= B\cap\lbrace Mf > \lambda\rbrace$.
Since the maximal function is lower semicontinuous, the set $E_\lambda$ is open; see \cite{MR2867756}*{Lemma 3.12}. 
For every $x\in E_\lambda$ consider the open ball $B(x,r_x)$ with 
$r_x = \dist\left( x, \lbrace Mf \leq \lambda \rbrace \right)>0$.
We claim that there exists a disjoint and countable subfamily $\lbrace B_i=B(x_i,r_i) \rbrace_i$ of $\lbrace B(x,r_x) \rbrace_{x\in E_\lambda}$ such that $E_\lambda \subset \bigcup_i 5 B_i$. 
Indeed, if $\sup_{x\in E_\lambda} r_x= \infty$, then it is enough to take a single ball $B(x,r_x)$ such that $x\in E_\lambda$ and $r_x>\rad(B)$. 
On the other hand, if $\sup_{x\in E_\lambda} r_x< \infty$, then the Vitali covering lemma (\cite{MR2867756}*{Lemma 1.7}) provides us with the desired subfamily.

Consider the collection $\lbrace B_i\rbrace_i$ with the above properties. Observe that for every $i$, it holds that $ 5B_i \cap \lbrace Mf \leq  \lambda \rbrace  \neq \emptyset$. Consequently, there exists $y\in 5 B_i$ with 
$$
\dashint_{5B_i} f \dd\mu \leq Mf(y)  \leq \lambda.
$$ 

On the other hand, it is immediate that $B_i \subset \lbrace Mf > \lambda \rbrace$ for every $i$. Bearing in mind all these observations, we obtain
\begin{align*}
\int_{E_\lambda} f \dd\mu &\leq \sum_i \int_{5 B_i} f \dd\mu = \sum_i \mu(5B_i) \dashint_{5 B_i} f \dd\mu \leq C_0(C_d) \lambda \sum_i \mu(B_i) \\
&= C_0 \lambda \mu \Bigl( \bigcup_i B_i \Bigr)  
\leq  C_0 \lambda \mu \left( \lbrace Mf > \lambda \rbrace \right).
\end{align*}
\end{proof}

To localize the distributional sets of the maximal function, we apply the following lemma. 
\begin{lemma}\label{localizationX2B}
Let $B\subset X$ be a ball, and $w\in L^1(B)$ a nonnegative function. 
There exists a constant $C_1=C_1(C_d)>1$ such that 
$$M( w \Chi_B ) \leq C_1 w_{B}$$ 
everywhere in $X \setminus 2 B$.
\end{lemma}
\begin{proof}[Proof of Lemma \ref{localizationX2B}]
Let $x\in X \setminus 2B$, and let $B' \subset X$ be a ball with $B \cap B' \neq \emptyset$ and $x\in B'$. Denote by $r$ and $r'$ the radii of $B$ and $B'$ respectively, and choose a point $y\in B \cap B'$. Then $r'\geq r/2$, since otherwise we would have $\dist(x,y) \leq \diam(B') \leq 2r' <r$ and, denoting by $z_B$ the center of $B$, 
$$
\dist(x,y) \geq \dist(x,z_B)-\dist(y,z_B) \geq 2r-r =r,
$$
which is a contradiction. The fact that $r'\geq r/2$ easily implies that $B \subset 5 B'$, and thus $\mu(B') \geq c(C_d) \mu (5 B') \geq c \mu(B)$.
By the definition of the maximal function we obtain
$$
M( w \Chi_B ) (x) = \sup_{ \substack{B' \ni x\\ B \cap B' \neq \emptyset}} \frac{1}{\mu(B')}\int_{B \cap B'} w \dd\mu \leq w(B)  \sup_{ \substack{B' \ni x\\ B \cap B' \neq \emptyset}} \frac{1}{\mu(B')} \leq \frac{w(B)}{c \mu(B)},
$$
whereby the statement is proven. 
\end{proof}

Combining the previous two lemmas we obtain an estimate, which can be regarded as a weak version of \ref{char-log}. 
\begin{lemma}\label{weaklog4B}
Assume that $w$ satisfies \ref{char-fujii} with the constant $C$. There exist constants $C_2=C_2(C_d)$ and $C_3=C_3(C_d,C)>0$ such that 
$$
\int_B w \log^+ \left( \frac{w}{C_2 w_{2B}} \right)\dd\mu \leq C_3 w(4B)
$$
for every ball $B$ with $4B \Subset \Omega$.
\end{lemma}
\begin{proof}[Proof of Lemma \ref{weaklog4B}]
By Lemma \ref{localizationX2B} there exists a constant $C_1=C_1(C_d)$ such that 
$$
M( w \Chi_B ) \leq C_1 w_{B}
$$ 
in $X\setminus 2B$.
Let $C_2=C_2(C_d)>0$ be such that $C_2 w_{2B} \geq C_1 w_B$, and let $\lambda > C_2 w_{2B}$.  
Lemma \ref{reverseweaktypeestimate} implies that
\begin{align*} 
\int_{B\cap \lbrace w >\lambda \rbrace} w \dd\mu 
&\leq \int_{B\cap \lbrace M(w \Chi_B) >\lambda \rbrace} w \dd\mu
\leq C_0 \lambda \mu \left( \lbrace M(w \Chi_B) > \lambda \rbrace \right)\\
&= C_0 \lambda \mu \left(2B\cap \lbrace M(w \Chi_B) > \lambda \rbrace \right).
\end{align*}
With these remarks, we obtain
\begin{align*}
\int_{2B} M(w \Chi_B) \dd\mu
&\geq \int_{C_2 w_{2B}}^\infty \mu \left(2B\cap \lbrace M(w \Chi_B) > \lambda \rbrace \right) \dd \lambda 
\geq \frac{1}{C_0} \int_{C_2 w_{2B}}^\infty \frac{1}{\lambda} \int_{B\cap \lbrace w>\lambda \rbrace} w \dd\mu \dd\lambda \\
&= \frac{1}{C_0} \int_{ B\cap \lbrace w >C_2w_{2B} \rbrace} w \int_{C_2w_{2B}}^{w(x)} \frac{1}{\lambda} \dd \lambda \dd \mu
= \frac{1}{C_0}\int_B w \log^+ \left( \frac{w}{C_2 w_{2B}} \right)\dd\mu.
\end{align*}
If $B$ is a ball with $4B \Subset \Omega$, by assumption
$$
\int_{2B} M(w \Chi_{2B}) \dd \mu \leq C w(4B),
$$ 
and we conclude that
$$
\int_B w \log^+ \left( \frac{w}{C_2 w_{2B}} \right) \dd\mu \leq  C_0 \int_{2B} M(w \Chi_{B} ) \dd\mu \leq C_0 \int_{2B} M(w \Chi_{2B} ) \dd\mu \leq C_0 C w(4B).
$$
\end{proof}

With the previous lemma in hand, we are able to prove \ref{char-a} with $4B$ instead of $2B$. 
\begin{lemma}\label{weaknondoubling4B}
For every $\widetilde{\eta}>0$ there exists $\widetilde{\varepsilon}>0$ such that for every measurable set $F \subset B$ with $4 B \Subset \Omega$ and $\mu(F) \leq \widetilde{\varepsilon} \mu(B)$, we have $w(F) \leq \widetilde{\eta} w(4B)$. 
\end{lemma}
\begin{proof}[Proof of Lemma \ref{weaknondoubling4B}]
Given $\widetilde{\eta}>0$, choose $\gamma>1$ large enough so that $C_3/\log(\gamma) \leq \widetilde{\eta}/2$, and $\widetilde{\varepsilon}>0$ small enough so that $\widetilde{\varepsilon} C_2 \gamma \leq \widetilde{\eta}/2$. Now let $B$ be a ball with $4B \Subset \Omega$, and $F\subset B$ a measurable set with $\mu(F) \leq \widetilde{\varepsilon} \mu(B)$. 
Then we have
$$
w\left( F \cap \lbrace w \leq \gamma C_2 w_{2B} \rbrace \right) \leq C_2 \gamma w_{2B}\mu(F) \leq \widetilde{\varepsilon} \mu(B) C_2 \gamma w_{2B} \leq \frac{\widetilde{\eta}}{2} w(2B) \leq \frac{\widetilde{\eta}}{2} w(4B).
$$
Furthermore, by Lemma \ref{weaklog4B}, we obtain
\begin{align*}
w\left( F \cap \lbrace w > \gamma C_2 w_{2B} \rbrace \right) 
&=\frac{1}{\log(\gamma)} \int_{F \cap \lbrace w > \gamma C_2 w_{2B} \rbrace} w \log(\gamma) \dd\mu\\
&\leq \frac{1}{\log(\gamma)} \int_{F \cap \lbrace w > \gamma C_2 w_{2B} \rbrace} w \log \left( \frac{w}{C_2 w_{2B}} \right) \dd\mu \\
&\leq \frac{1}{\log(\gamma)} \int_{F \cap \lbrace w >  C_2 w_{2B} \rbrace} w \log \left( \frac{w}{C_2 w_{2B}} \right) \dd\mu\\
&\leq \frac{C_3 w(4B)}{\log(\gamma)} \leq \frac{\widetilde{\eta}}{2} w(4B).
\end{align*}
Combining both estimates we conclude that $w(F) \leq \widetilde{\eta} w(4B)$. 
\end{proof}

Finally, we are set to complete the proof of \ref{char-fujii} $\implies$ \ref{char-a}. To this end, let $B$ be a ball such that $2B \Subset \Omega$, and let $F \subset B$ be a measurable set. Let $\lbrace B_i \rbrace_{i=1}^N$ be the collection of balls from Lemma \ref{claimsigmacoveringballs} with $\sigma=1/5$. Observe that $4 B_i \subset 2 B \Subset \Omega$ and that $\mu(B_i) \geq \widetilde{c}(C_d) \mu(B)$, because the center of each $B_i$ is contained in $B$. Also note that the number $N$ of balls only depends on $C_d$. 

For any $\eta>0$ denote $\widetilde{\eta} = N \eta$, and let $\widetilde{\varepsilon}$ be the parameter associated with $\widetilde{\eta}$ from Lemma \ref{weaknondoubling4B}. Let $\varepsilon= \widetilde{c}(C_d) \widetilde{\varepsilon}$. If $\mu(F) \leq \varepsilon \mu(B)$, then
$$
\frac{\mu(F \cap B_i)}{\mu(B_i)} \leq \frac{\mu(F)}{\widetilde{c}(C_d) \mu(B)} \leq \widetilde{\varepsilon},
$$
which by Lemma \ref{weaknondoubling4B} implies that $w( F \cap B_i) \leq \widetilde{\eta} w( 4B_i) \leq \widetilde{\eta} w( 2B)$ for every $i=1, \ldots, N$. We conclude that
$$
w(F) \leq \sum_{i=1}^N w( F \cap B_i ) \leq \widetilde{\eta} \sum_{i=1}^N w( 4B_i)  \leq \widetilde{\eta} N w(2B) = \eta w(2B).
$$
\end{proof}

\begin{propo}
\ref{char-b} $\iff$ \ref{char-dist}
\end{propo}
\begin{proof}
To show \ref{char-b} $\implies$ \ref{char-dist}, let $\alpha=\varepsilon^{-1} C_d$, $\beta =\eta$, and $F= B\cap\left\lbrace w\geq  \alpha w_{2B} \right\rbrace$. 
We may assume that $w(2B)>0$, since otherwise the estimate is trivial. We thus have
$$
w(2B) \geq w(F) \geq \alpha w_{2B} \mu(F),
$$
which implies that 
$$
\mu(F) \leq \alpha^{-1} \mu(2B) \leq  \alpha^{-1} C_d \mu(B) = \varepsilon \mu(B).
$$ 
By \ref{char-b} we obtain $w(F) \leq \eta w(2B) = \beta w(2B)$. 

As for the reverse implication, let $\varepsilon= (2\alpha)^{-1}\left(C_d^{-5}-\beta\right)$ and $\eta= 2^{-1}\left(C_d^{-5}+\beta\right)$. Notice that $\varepsilon, \eta >0$, and $\eta< C_d^{-5}$. 
For a measurable set $F\subset B$ with $\mu(F) \leq \varepsilon \mu(B)$, we obtain
\begin{align*}
w(F) & \leq w\left(F\cap \left\lbrace w\geq \alpha w_{2B} \right\rbrace \right) + w\left(F\cap \left\lbrace w\leq \alpha w_{2B} \right\rbrace \right)  \\
& \leq 
w\left(B\cap \left\lbrace w\geq  \alpha w_{2B} \right\rbrace \right)
 +  \alpha \mu(F) w_{2B}\\
 & \leq \beta w(2B) + \alpha \varepsilon w(2B) = \eta w(2B).
\end{align*} 
\end{proof}

\begin{propo}\label{CharLog}
\ref{char-d} $\implies$ \ref{char-log} $\implies$ \ref{char-b}
\end{propo}
\begin{proof}
The proof of \ref{char-log} $\implies$ \ref{char-b} is identical to that of \ref{char-hmm} $\implies$ \ref{char-b} (Proposition \ref{logimplieshmmimpliesb}) with $\phi = \log$. We show the other implication. For every ball $B$ and denote $F_t=B\cap\left\lbrace w> t \right\rbrace$, with $t>0$.  By Fubini's theorem
\begin{align*} 
\int_B w \log^+\left(\frac{w}{w_{2B}} \right) \dd \mu 
&= \int_{F_{w_{2B}}} w \log\left(\frac{w}{w_{2B}} \right) \dd \mu
= \int_{F_{w_{2B}}} w(x) \int_{w_{2B}}^{w(x)} \frac{dt}{t} \dd \mu(x)\\
&= \int_{w_{2B}}^\infty \frac{1}{t}  \int_{F_t} w \dd \mu  \dd t.
\end{align*}
By the assumption \ref{char-d}, there exist constants $C>1$, $\alpha>0$ such that 
$$
w(F_t) \leq C w(2B) \left( \frac{\mu(F_t)}{\mu(B)} \right)^\alpha.
$$ 
Also, the inequality $t \mu(F_t)\leq w(F_t) \leq w(2B)$ holds for every $t>0$. With these observations, we have
\begin{align*}
\int_{w_{2B}}^\infty \frac{1}{t}  \int_{F_t} w \dd \mu  \dd t
&\leq \int_{w_{2B}}^\infty \frac{C}{t} w(2B) \left( \frac{\mu(F_t)}{\mu(B)} \right)^\alpha \dd t 
\leq \frac{w(2B)^{1+\alpha}}{\mu(B)^\alpha}\int_{w_{2B}}^\infty \frac{C}{t^{1+\alpha}} \dd t \\
&= \frac{C}{\alpha} \left( \frac{\mu(2B)}{\mu(B)} \right)^\alpha w(2B) \leq \frac{C C_d^\alpha}{\alpha}  w(2B).
\end{align*}

\end{proof}

\begin{propo}
\ref{char-d} $\implies$ \ref{char-saw1} $\implies$ \ref{char-a}
\end{propo}
\begin{proof}
The first implication is immediate by setting $\phi = \left(\cdot\right)^\alpha$. As for \ref{char-saw1} $\implies$ \ref{char-a}, for every $\eta>0$, we choose $\varepsilon>0$ with $\phi(\varepsilon) \leq \eta$. Let $F \subset B$ be a measurable set with $\mu(F) \leq \varepsilon \mu(B)$. Because $\phi$ is nondecreasing we have 
\begin{equation*}
\frac{w(F)}{w(2B)} \leq \phi\left(\frac{\mu(F)}{\mu(B)}\right) \leq \phi\left(\frac{\varepsilon\mu(B)}{\mu(B)}\right) \leq \phi\left(\varepsilon\right) \leq \eta.
\end{equation*}
\end{proof}

\begin{propo}
\ref{char-c} $\implies$ \ref{char-saw2} $\implies$ \ref{char-a}
\end{propo}
\begin{proof}
Assume that $w$ satisfies \ref{char-c}. Let $f\in\BMO(\Omega)$ such that $\norm{f}_{\BMO(\Omega)} \leq 1$, and $p$ such that \ref{char-c} holds. Let $B $ be a ball with $11 B\Subset \Omega$. Using H\"older's inequality
\begin{align*}
\int_B\abs{f-f_B}w\dd\mu &\leq \left(\int_B\abs{f-f_B}^{p'}\dd\mu\right)^\frac{1}{p'}\left(\int_Bw^p\dd\mu\right)^\frac{1}{p}\\
&= \mu(B) \left(\dashint_B\abs{f-f_B}^{p'}\dd\mu\right)^\frac{1}{p'}\left(\dashint_Bw^p\dd\mu\right)^\frac{1}{p}.
\end{align*}
The John-Nirenberg inequality for $\BMO$ functions (Proposition 3.19 in \cite{MR2867756}) implies that
$$
\left(\dashint_B\abs{f-f_B}^{p'}\dd\mu\right)^\frac{1}{p'}
\le C(p,C_d)\norm{f}_{\BMO(11B)}.
$$
By \ref{char-c} we obtain
\begin{align*}
\int_B\abs{f-f_B}w\dd\mu 
&\le C\mu(B)\norm{f}_{\BMO(11B)}\dashint_{2B}w\dd\mu \\
&=C\mu(B)\norm{f}_{\BMO(11B)}\frac{w(2B)}{\mu(2B)} \leq Cw(2B).
\end{align*}

Let us now assume \ref{char-saw2}. To begin with, we make the following claim.
\begin{claim}\label{claimlogarithmmaximal}
Let $F$ be a measurable subset of a ball $B \subset X$ with $\mu(F)>0$. Then there exists a constant $A>1$ depending only on $C_d$ such that the function $f = \log^+\left( \mu(B)\mu(F)^{-1} M( \Chi_F) \right)$ belongs to $\BMO(X)$ with $\norm{f}_{\BMO(X)} \leq A$. 
\end{claim}

\begin{proof}[Proof of Claim \ref{claimlogarithmmaximal}]
By the definition of $\BMO$, it is enough to verify that the function $g = \frac{1}{2}\log\left( M(\Chi_F) \right)$ is in $\BMO(X)$ with norm bounded by a constant only depending on $C_d$. The Coifman--Rochberg theorem states that the function $w = M(\Chi_F)^{1/2}$ is a Muckenhoupt weight of class $A_1(X)$, with its charateristic constant bounded by a constant $C(C_d)$ that depends only on $C_d$. Therefore, we have $\log(w) \in \BMO(X)$ with $$\norm{\log(w)}_{\BMO(X)} \leq \log \left( 2 C(C_d) \right).$$ These are well-known results from the Euclidean theory and have been shown e.~g. in \cite{MR807149}. Those interested in detailed proofs in a metric space may consult \cites{MR4340793} and \cite{shukla_thesis}, respectively.
\end{proof}

We will be needing the following weak version of \ref{char-a}.
\begin{lemma}\label{lemmaibfor11B}
For every $\widetilde{\eta}>0$, there exists an $\widetilde{\varepsilon}>0$ such that for every $F \subset B$ with $11 B \Subset \Omega$ and $\mu(F) \leq \widetilde{\varepsilon} \mu(B)$ we have $w(F) \leq \widetilde{\eta} w(2B)$. 
\end{lemma}
\begin{proof}[Proof of Lemma \ref{lemmaibfor11B}]
Given $\widetilde{\eta}>0$, we choose $\widetilde{\varepsilon}>0$ small enough such that $ \widetilde{C} \sqrt{\widetilde{\varepsilon}} \leq 1/4 $ and $\widetilde{\varepsilon} \leq e^{-4 A C/\widetilde{\eta}}$, where $C$ is given by \ref{char-saw2}, $A=A(C_d)$ is the constant from Claim \ref{claimlogarithmmaximal}, and $\widetilde{C}=\widetilde{C}(C_d)$ is the bound for the $(1,1)$-weak type inequality (Theorem \ref{thm:weaktype}) for the maximal function. Furthermore, let $0 < \delta = \mu(F)\mu(B)^{-1} \leq \widetilde{\varepsilon}<1$, $r= \sqrt{\delta}$, and $f$ be the function from Claim \ref{claimlogarithmmaximal}. Then
\begin{align*}
f_B &\leq \frac{1}{\mu(B)} \int_{B \cap \lbrace M(\Chi_F) \leq r \rbrace} f \dd \mu + \frac{1}{\mu(B)} \int_{B \cap \lbrace M(\Chi_F) > r \rbrace} f \dd \mu \\
&\leq \log^+\left( \frac{\mu(B)}{\mu(F) }  r \right)+ \frac{1}{\mu(B)}\log^+\left( \frac{\mu(B)}{\mu(F) } \right)  \mu \left(\lbrace M(\Chi_F) > r \rbrace \right) \\
&\leq \log^+\left( \frac{\mu(B)}{\mu(F) }  r \right)+ \frac{1}{\mu(B)}\log^+\left( \frac{\mu(B)}{\mu(F) } \right) \frac{\widetilde{C} }{r} \int_X \Chi_F \dd \mu\\
&= \log\left( \frac{1}{\sqrt{\delta}} \right) + \widetilde{C} \sqrt{\delta } \log\left( \frac{1}{ \delta} \right).
\end{align*}
By the choice of $\widetilde{\varepsilon}$ and the fact that $\delta\leq \widetilde{\varepsilon}$, it follows that
\begin{equation}\label{estimatedifferencelogarithmaverage}
\log\left( \frac{\mu(B)}{\mu(F) } \right) - f_B \geq \log\left( \frac{1}{ \delta} \right)- \log\left( \frac{1}{\sqrt{\delta}} \right) - \widetilde{C} \sqrt{\delta } \log\left( \frac{1}{ \delta} \right) \geq \frac{1}{4} \log\left( \frac{1}{ \delta} \right) \geq \frac{1}{4} \log\left( \frac{1}{ \widetilde{\varepsilon} }\right).
\end{equation}

On the other hand, notice that $f= \log\left(\mu(B)\mu(F)^{-1} \right)$ almost everywhere on $F$. Thanks to Claim \ref{claimlogarithmmaximal} we have $\norm{f}_{\BMO(X)} \leq A$, and thus \ref{char-saw2} for $f/A$ gives
$$
\left( \log\left( \frac{\mu(B)}{\mu(F) } \right) - f_B  \right) w (F) \leq \int_F | f-f_B| w \dd \mu \leq  \int_B | f-f_B| w \dd \mu \leq A C w(2B).
$$
Combining this estimate with \eqref{estimatedifferencelogarithmaverage}, we conclude that
$$
w(F) \leq \frac{4 A C}{\log(\widetilde{\varepsilon}^{-1})} w(2B) \leq \widetilde{\eta} w(2B).
$$
\end{proof}

It remains to complete the proof of \ref{char-saw2} $\implies$ \ref{char-a}, which is nearly the same as the final step of Proposition \ref{charc}. Let $B $ a ball with $2B \Subset \Omega$ and $F \subset B$ a measurable set. Let $\lbrace B_i \rbrace_{i=1}^N$ be the collection of balls from Lemma \ref{claimsigmacoveringballs} with $\sigma=1/12$. Observe that $11 B_i \subset 2 B \Subset \Omega$ and that $\mu(B_i) \geq \widetilde{c}(C_d) \mu(B)$ because the center of each $B_i$ is contained in $B$. Also notice that the number of balls $N$ only depends on $C_d$. 

For any $\eta >0$ define $\widetilde{\eta} = N \eta$, and let $\widetilde{\varepsilon}$ be the parameter associated with $\widetilde{\eta}$ in Lemma \ref{lemmaibfor11B}. Let $\varepsilon= \widetilde{c}(C_d) \widetilde{\varepsilon}$. If $\mu(F) \leq \varepsilon \mu(B)$, then
$$
\frac{\mu(F \cap B_i)}{\mu(B_i)} \leq \frac{\mu(F)}{\widetilde{c}(C_d) \mu(B)} \leq \widetilde{\varepsilon},
$$
implying by Lemma \ref{lemmaibfor11B} that $w( F \cap B_i) \leq \widetilde{\eta} w( 2B_i)$ for every $i=1, \ldots, N$. We conclude that
$$
w(F) \leq \sum_{i=1}^N w( F \cap B_i ) \leq \widetilde{\eta} \sum_{i=1}^N w( 2B_i) \leq \widetilde{\eta} \sum_{i=1}^N w( 11 B_i) \leq \widetilde{\eta} N w(2B) = \eta w(2B).
$$
\end{proof}

\begin{propo}\label{logimplieshmmimpliesb}
\ref{char-log} $\implies$ \ref{char-hmm} $\implies$ \ref{char-b}
\end{propo}
\begin{proof}
To see that \ref{char-log} implies \ref{char-hmm}, it is enough to choose $\phi= \log$. We show the other implication. If $C>0$ and $\phi$ are as in \ref{char-hmm}, take $\gamma >1$ large enough so that $\phi(\gamma) > 2 C C_d^{5}$. Let $\varepsilon= (2\gamma C_d^{5})^{-1}$ and 
$$\eta = \varepsilon \gamma+ \frac{C}{\phi(\gamma)}.$$
Notice that $\eta < C_d^{-5}$. 
Let $F \subset B$ be such that $\mu(F) \leq \varepsilon \mu(B),$ and define $F_1=F\cap\lbrace w\leq \gamma w_{2B} \rbrace$, $F_2= F \setminus F_1$. It is immediate that $w(F_1) \leq \gamma w_{2B} \mu(F_1)$. Using the assumption \ref{char-hmm} and the fact that $\phi$ is nondecreasing, we obtain
\begin{align*}
w(F_2) &=  \frac{1}{\phi(\gamma)}\int_{F\cap \lbrace w> \gamma w_{2B} \rbrace } w \phi(\gamma) \dd\mu 
\leq \frac{1}{\phi(\gamma)}\int_{F\cap \lbrace w > \gamma w_{2B} \rbrace } w \phi\left( \frac{w}{w_{2B}} \right) \dd\mu \\
&\leq \frac{1}{\phi(\gamma)}\int_{B\cap\lbrace w >  w_{2B} \rbrace } w \phi\left( \frac{w}{w_{2B}} \right) \dd\mu
\leq \frac{C w(2B)}{\phi(\gamma)}.
\end{align*}
Collecting the estimates for $w(F_1)$ and $w(F_2)$ gives
$$
\frac{w(F)}{w(2B)} \leq \frac{\mu(F)}{\mu(2B)} \gamma + \frac{C}{\phi(\gamma)} \leq \varepsilon \gamma + \frac{C}{\phi(\gamma)} = \eta.
$$
\end{proof}

Although the assertions in Theorem \ref{thm:char} are modeled after conditions that hold true for Muckenhoupt weights, 
we briefly discuss two conditions for Muckenhoupt $A_\infty$ weights that fail to hold for functions satisfying a weak reverse H\"older inequality.

\begin{enumerate}[label=\normalfont{(*\alph*)}]
\item\label{char-expfail} (Exponential-type condition) There exists a constant $C>0$ such that 
$$
\dashint_{B} w \dd \mu \leq  C \exp \left(\dashint_{2B}\log w  \dd\mu\right)
$$
for every ball $B$ with $2B\Subset \Omega$.
\item\label{char-subfail} (Sublevel sets condition) There exist constants $0<\alpha, \beta<1$ such that 
$$
\mu\left(B\cap \left\lbrace w\leq \beta w_{2B} \right\rbrace \right) \leq \alpha \mu(B)
$$
for every ball $B$ with $2B\Subset \Omega$.
\end{enumerate}

For a counterexample, let $w(x)=e^x$ on $\R$. By Example \ref{example:revclass}, this is indeed a weight satisfying the assertions in Theorem \ref{thm:char}. However, $w$ does not satisfy \ref{char-expfail}. To see this, assume that there exist  constants $C>0,$ $\alpha \in \R$ such that for every interval $B\subset\R$,
$$
\dashint_B w\dd x\leq C \exp \left( \alpha \dashint_{2B} \log w\dd x \right).
$$
Consider intervals $B=(-r,r)$ with $r>0$ centered at the origin.  Then 
$$
\dashint_B w\dd x=\frac{e^{r}-e^{-r}}{2r}, \quad \dashint_{2B} \log w \dd x= \frac{1}{4r}\int_{-2r}^{2r} x \dd x= 0.
$$
By assumption, this means that $\left(e^{r}-e^{-r}\right)\left(2r\right)^{-1} \leq C \exp(0)=C$ for every $r>0$, a contradiction because the left-hand side is an unbounded function of $r>0$.

In the same way, assume that there exist constants $0<\alpha, \beta<1$ such that \ref{char-subfail} holds. The condition on the left-hand side becomes
$$
\beta \dashint_{2B}w\dd x = \frac{\beta}{4r}\left(e^{2r}-e^{-2r}\right) 
$$
which for large enough $r$ means that $\abs{ B \cap \lbrace w \leq \beta w_{2B} \rbrace }=\abs{B}$, whereby the claim becomes $\abs{B}\leq \alpha\abs{B}< \abs{B}$, a contradiction. 

\section{Qualitative characterization and the reverse H\"older inequality}\label{weakRHI}

In this section we show that the assertions \ref{char-a} -- \ref{char-c} in Theorem \ref{thm:char} are equivalent. 
This is the content of the following theorem, which is the generalization of \cite{spadaro}*{Theorem 1.1} to metric measure spaces with a doubling measure. 
\begin{theorem}\label{maintheorem}
Let $(X,d,\mu)$ be a metric measure space with a doubling measure. Assume that $\Omega \subset X$ is an open set, and let $w$ be a weight on $\Omega$. The following statements are equivalent.
\begin{enumerate}[label=\normalfont{(\roman*)}]
\item\label{one} There exist $p>1$ and a constant $C>0$ such that 
$$
\dashint_B w^p \dd \mu \leq C \left(  \dashint_{2B} w \dd \mu \right)^p
$$
for every ball $B$ with $2B\Subset \Omega$.
\item\label{two} For every $\eta>0$, there exists an $\varepsilon >0$ such that if $B$ is ball with $2B \Subset \Omega$ and $F \subset B$ is a measurable set, then
$\mu(F) \leq \varepsilon \mu(B)$ implies that $w(F) \leq \eta w(2B)$.
\end{enumerate}
\end{theorem}
The implication \ref{one} $\Rightarrow$ \ref{two} is a simple consequence of H\"older's inequality. In fact, this argument allows to show a quantitative version of \ref{two}. We give a full proof of the reverse implication \ref{two} $\Rightarrow$ \ref{one}, where it is in fact enough to assume that there exist positive $\eta, \varepsilon$ with $\eta < C_d^{-5}$ for which the condition holds, which corresponds to Theorem \ref{thm:char} \ref{char-b}. The upper bound for $\eta$ tends to zero with increasing dimension and cannot be done away with, unlike in the case of Muckenhoupt weights where any $0 < \eta, \varepsilon <1$ will suffice. The following example was given by Sawyer with $n=2$ in \cite{MR654182}.
\begin{example}
Consider $\R^n$ equipped with the Lebesgue measure. Let $S=\lbrace (x_1,\ldots, x_n) \in \R^n \mathbin{:} 0 \leq x_n \leq 1 \rbrace$, and $w=\Chi_S$. For every cube $Q \subset \R^n$ with $\mu(Q \cap S) >0$, the sets $Q \cap S$ and $2Q \cap S$ are rectangles in $\R^n$ whose first $n-1$ sides have length equal to $l(Q)$ and $2l(Q)$ respectively, with $l$ denoting side length. The length of the $n$th side of $Q \cap S$ is no greater than the length of the $n$th side of $2 Q \cap S$. Hence
$$
\frac{w(Q)}{w(2Q)} = \frac{\mu( Q \cap S)}{\mu( 2Q \cap S)} \leq \frac{1}{2^{n-1}}
$$
for every cube $Q$. In particular, for every cube $Q \subset \mathbb{R}^n$ and every measurable subset $F\subset Q,$ we have $w(F) \leq 2^{1-n} w(2Q).$ However, $w$ does not belong to $WRH_p(\R^n)$ for any $p>1$. In other words, $w$ does not satisfy any of the assertions in Theorem \ref{thm:char}. Indeed, considering cubes $Q=\left[-r/2,r/2\right]^n$ centered at the origin with side length $r \geq 2$, and $F = Q \cap S$, we note that
$$
\mu(Q)= r^n, \quad \mu(F) = w(F)= \mu(Q \cap S)=r^{n-1}\quad\text{and}\quad w(2Q)=(2r)^{n-1}.
$$
Theorem \ref{thm:char}~\ref{char-d}, if valid for $w$, would give constants $c, \alpha>0$ such that 
$$
\frac{1}{2^{n-1}}= \frac{r^{n-1}}{(2r)^{n-1}}=\frac{w(F)}{w(2Q)} \leq c \left( \frac{\mu(F)}{\mu(Q)} \right)^\alpha =  \frac{c}{r^\alpha}
$$
for every $r \geq 2$, which is a contradiction. 

The example shows that the upper bound for $\eta$ in the characterization of weak $A_\infty$ weights given by Theorem \ref{thm:char} \ref{char-b} must be smaller than $2^{1-n}.$
\end{example}

The proof of Theorem \ref{maintheorem} relies on the following lemma.
\begin{lemma}\label{mainlemma}
Assume that $w$ satisfies \ref{two} of Theorem \ref{maintheorem}. There exist  constants $\gamma >C_d^3$, and $ \beta>0$, only depending on the parameters of \ref{two}, for which the following statement holds. Let $B$ be a ball with $2B \Subset \Omega$,  $0<r<3/2$, and $\lambda < 10^{-1}$. Then 
$$
\int_{r B \cap \lbrace w \geq \gamma D \rbrace } w \dd \mu \leq \gamma^{-\beta} \int_{(r+\lambda) B \cap \lbrace w \geq \gamma^{-1} D \rbrace } w \dd \mu, 
$$
where 
$$
D=D(\lambda, B)=\frac{w(2B)}{\mu(2B)} C_d^{\log_2 \left( \frac{4}{5\lambda} \right) +1 }.
$$ 
\end{lemma}
\begin{proof}
For an $\eta >0$ small enough such that $ \eta C_d^5 < 1$, let $\varepsilon$ be the parameter associated with $\eta$ from \ref{two}. We choose the constants $\gamma $ and $\beta>0$ so that
\begin{equation}\label{choicegammabeta}
\gamma > \max \left\lbrace C_d^3, \frac{C_d^2}{\varepsilon}, \frac{1}{1-\eta C_d^5} \right\rbrace,
\quad\text{and}\quad 
\gamma^{\beta}=\frac{1-\gamma^{-1}}{\eta C_d^5}.
\end{equation}
Fix $B$, $r$, $\lambda$, and $D=D(\lambda,B)$ as in the assumption. 
We denote $A= \gamma D$, $a= \gamma^{-1} D$, 
$$I=\lbrace y\in \Omega \mathbin{:}w(y) \geq A \rbrace,
\quad\text{and}\quad J=\lbrace y\in \Omega \mathbin{:}w(y) \geq a \rbrace.
$$ 
We may and do assume that $\mu\left( r B \cap I \right)>0,$ as otherwise the inequality trivially holds. Let $x$ be a Lebesgue point of $w$ contained in $r B \cap I$. The inclusions $B(x,5\lambda \rad(B)) \subset 2B \subset B(x, 4 \rad(B))$ together with the doubling condition give 
\begin{equation}\label{sx}
\dashint_{B(x,5\lambda \rad(B))} w\dd\mu \leq \frac{w(2B)}{\mu\left( B(x,5\lambda \rad(B)) \right)} \leq \frac{w(2B)}{\mu(2B)}  C_d^{\left\lfloor\log_2 \left( \frac{4}{5\lambda} \right)\right\rfloor +1} \leq D.
\end{equation}
Denote 
$$
s_x =\inf \left\lbrace s>0 \mathbin{:} \:B(x,s) \Subset \Omega \: \text{ and } \: \dashint_{B(x,s)} w\dd\mu \leq D \right\rbrace,
\quad\text{and}\quad 
r_x= \frac{s_x}{10}.
$$
It follows from \eqref{sx} that $s_x \leq 5 \lambda \rad(B)$. Also, because $x$ is a Lebesgue point with $w(x) \geq \gamma D > D$, it is clear that $s_x>0$. Thus there exists a number $\widetilde{s}_x$ such that $s_x\leq \widetilde{s}_x \leq 2 s_x, $ $B(x, \widetilde{s}_x) \Subset \Omega,$ and
$$
\dashint_{B(x,\widetilde{s}_x)} w\dd\mu  \leq D.
$$ 

We have
\begin{align*}
\mu\left( I \cap B(x, 10r_x) \right)  &= \mu\left( I \cap B(x, s_x)  \right) \leq A^{-1} \int_{B(x,s_x)} w \dd\mu \leq A^{-1} \int_{B(x,\widetilde{s}_x)} w \dd\mu \\ 
& \leq  A^{-1} D \mu \left( B(x,\widetilde{s}_x) \right)  \leq  A^{-1} D \mu \left( B(x,2 s_x) \right) =A^{-1} D \mu \left( B(x,20 r_x) \right).
\end{align*}
The doubling condition and \eqref{choicegammabeta} imply that
\begin{align*}
\mu\left( I \cap B(x, 5r_x) \right)&\leq \mu\left( I \cap B(x, 10r_x) \right) \leq A^{-1} D C_d^2 \mu(B(x, 5 r_x))\\
&\leq \gamma^{-1} C_d^2 \mu(B(x, 5 r_x)) \leq \varepsilon \mu(B(x, 5 r_x)).
\end{align*}
This, in turn, lets us apply the assumption \ref{two} to estimate
\begin{align}
\nonumber \int_{I \cap B(x, 5 r_x)} w \dd\mu  &\leq \eta \int_{B(x,10r_x)} w \dd\mu =\eta \int_{B(x,s_x)} w \dd\mu \leq \eta \int_{B(x,\widetilde{s}_x)} w \dd\mu \leq \eta D \mu \left( B(x,\widetilde{s}_x) \right) \\
&\leq \eta D \mu \left( B(x,2 s_x) \right) \leq \eta D C_d^5  \mu \left( B(x,r_x) \right) \leq  \eta C_d^5 \int_{B(x,r_x)} w \dd\mu, \label{estimateI}
\end{align}
where the last inequality follows from the fact that $r_x$ is smaller than the infimum in the definition of $s_x$. We use again the fact that $r_x$ is smaller than $s_x$ to deduce
\begin{align*}
\int_{B(x,r_x)} w \dd\mu &= \int_{B(x,r_x)\setminus J} w \dd\mu + \int _{J \cap B(x,r_x)} w \dd\mu  \leq a \mu \left( B(x,r_x) \right) + \int _{J \cap B(x,r_x)} w \dd\mu\\
&\leq \frac{a}{D} \int_{B(x,r_x)} w \dd\mu +  \int _{J \cap B(x,r_x)} w \dd\mu,
\end{align*}
which implies that 
\begin{equation}\label{estimateJ}
\int_{B(x,r_x)} w \dd\mu \leq \left(1- \frac{a}{D}\right)^{-1} \int _{J \cap B(x,r_x)} w \dd\mu = \left(1- \gamma^{-1} \right)^{-1} \int _{J \cap B(x,r_x)} w \dd\mu.
\end{equation}
Inserting \eqref{estimateJ} into \eqref{estimateI} and recalling the choice of the parameters $\eta$, $\gamma$, and $\beta$ \eqref{choicegammabeta}, we conclude that 
\begin{equation}\label{estimatelocalballs}
\int_{I \cap B(x, 5 r_x)} w \dd\mu \leq \frac{\eta C_d^5}{1- \gamma^{-1}} \int _{J \cap B(x,r_x)} w \dd\mu = \gamma^{-\beta} \int _{J \cap B(x,r_x)} w \dd\mu
\end{equation}
for every Lebesgue point of $w$ contained in $I \cap r B$. If $F$ denotes the set of these points, then the Vitali covering lemma (\cite{MR2867756}*{Lemma 1.7}) provides us with a collection $\lbrace x_j \rbrace \subset F$ such that $\bigcup_{x\in F} B(x,r_x) \subset \bigcup_{j} B(x_j,5 r_j)$ and the family of balls $\lbrace B(x_j,r_j) \rbrace_j$ is disjoint, where we have written $r_j=r_{x_j}$ for short.  Note that we are allowed to apply Vitali lemma because, for every $x\in F,$ we have $r_x \leq \lambda r(B)/2.$ Since these balls cover almost every point in $I \cap r B$, the estimate \eqref{estimatelocalballs} implies that
$$
\int_{I \cap r B } w \dd\mu \leq \sum_{j} \int_{I \cap B(x_j, 5r_j)} w \dd\mu \leq \gamma^{-\beta} \sum_j \int _{J \cap B(x_j,r_j)} w \dd\mu 
= \gamma^{-\beta} \int_{ J \cap \left( \bigcup_j B(x_j,r_j) \right)} w \dd\mu.
$$
Finally, observe that $r_x=s_x/10 \leq \lambda \rad(B)/2$ for every $x \in F$,
which implies that $B(x,r_x) \subset (r + \lambda) B$, and thus 
$$
\gamma^{-\beta} \int_{ J \cap \left( \bigcup_j B(x_j,r_j) \right)} w \dd\mu
\le\gamma^{-\beta} \int_{ J \cap \left( r+\lambda \right) B} w\dd\mu.
$$ 
\end{proof}

\begin{proof}[Proof of Theorem \ref{maintheorem}]
We show that \ref{two} $\implies$ \ref{one}. Let $B$ be a ball with $2B\Subset \Omega$.
Let $\gamma$ and $\beta$ are as in Lemma \ref{mainlemma}, and $p>1$ so that $2(p-1)<\beta.$ Denote
$$
\lambda_k = \frac{4}{5}\cdot 2^{1- 2k \log(\gamma)/\log(C_d)}, \quad k=1,2,\dots.
$$
Since $\log(\gamma) \geq 3 \log(C_d)$, it is easy to verify that $\lambda_k < 1/10 $ for every $k=1,2,\dots$ and that $\sum_{k=1}^\infty \lambda_k < 1/2$. Also, the corresponding constants $D_k=D(\lambda_k, B)$ from the statement of Lemma \ref{mainlemma} satisfy
\begin{equation}\label{definitionsDk}
D_k = \gamma^{2k} \frac{w(2B)}{\mu(2B)},\quad k=1,2,\dots, \quad \text{and} \quad D_k= \gamma^2 D_{k-1},\quad k =2,3,\dots.
\end{equation}
Then
\begin{align}\label{withoutaverages}
\nonumber\int_B w^p \dd\mu &= \int_{B \cap \lbrace w \leq \gamma D_1 \rbrace} w^p \dd\mu + \sum_{k=1}^\infty \int_{B \cap \lbrace \gamma D_k\leq w \leq \gamma D_{k+1} \rbrace} w^{p-1} w\dd\mu \\
& \leq \gamma^p D_1^p \mu(B)+ \sum_{k=1}^\infty  \left(\gamma D_{k+1} \right)^{p-1}\int_{B \cap \lbrace w \geq \gamma D_{k} \rbrace}  w \dd\mu.
\end{align}
For $k=1,2,\dots$, we apply Lemma \ref{mainlemma} repeatedly together with \eqref{definitionsDk} to obtain
\begin{align*}
 \int_{B \cap \lbrace w \geq \gamma D_{k} \rbrace} w\dd\mu
&\leq \gamma^{-\beta} \int_{(1+\lambda_k) B \cap \lbrace w \geq \gamma D_{k-1} \rbrace}  w \dd\mu \\
&\leq \gamma^{-2\beta} \int_{(1+\lambda_k+\lambda_{k-1}) B \cap \lbrace w \geq \gamma D_{k-2} \rbrace} w \dd\mu \\
&\leq \cdots 
\leq \gamma^{-(k-1)\beta} \int_{(1+\lambda_k+\cdots +\lambda_{1}) B \cap \lbrace w \geq \gamma D_{1} \rbrace}  w \dd\mu \\
&\leq \gamma^{-(k-1)\beta} \int_{\frac{3}{2}B} w \dd\mu
\leq \gamma^{-(k-1)\beta} w(2B).
\end{align*}
Combining this with \eqref{withoutaverages}, we have
\begin{align*} 
\dashint_B w^p \dd\mu 
&\leq \gamma^{3p} \left( \frac{w(2B)}{\mu(2B)} \right)^p+ \frac{1}{\mu(B)} \sum_{k=1}^\infty \gamma^{(2k+3)(p-1)} \left( \frac{w(2B)}{\mu(2B)} \right)^{p-1} \gamma^{-(k-1)\beta} w(2B) \\
&\leq \gamma^{3p} \left( \frac{w(2B)}{\mu(2B)} \right)^p + C_d \left( \frac{w(2B)}{\mu(2B)} \right)^{p} \gamma^{5(p-1)} \sum_{k=1}^\infty \gamma^{(2(p-1)-\beta)(k-1)}.
\end{align*}
By the choice of $p$, the series above converges and there exists a constant $C$ depending on the parameters in \ref{two}, as well as on $C_d,\gamma, \beta$, and $p$, such that
$$
\dashint_B w^p \dd \mu \leq C \left( \frac{w(2B)}{\mu(2B)} \right)^p.
$$
\end{proof}

We remark that the factor $2$ in Theorem \ref{thm:char} can be replaced with any other $\sigma>1$, resulting in the same class of weak $A_\infty$ weights.  In the particular case where $\Omega = X$, this formulation coincides with the $\sigma$-weak reverse H\"older classes of weights introduced by Anderson, Hyt\"onen, and Tapiola in \cite{MR3606546} and denoted by them by $RH_p^\sigma$. Anderson et~al. show that $RH_p^\sigma = RH_p^{\sigma'}$ for every $\sigma, \sigma',p>1$; moreover, $w\in RH_p^\sigma$ is equivalent to
$$
\sup_{B \subset X} \frac{1}{w(\sigma B)} \int_B M(w \Chi_{B}) \dd \mu < \infty.
$$

Theorem \ref{maintheoremwithsigma} for an arbitrary $\sigma$ is stated below without proof. Most statements follow by imitating the proof of Theorem \ref{thm:char} together with a covering argument such as that of Lemma \ref{claimsigmacoveringballs}. Namely, for every $\varepsilon>0$ and every ball $B$, we can cover $B$ with $N = N(\varepsilon, C_d)$ balls centered at points of $B$ and with radius $\varepsilon \rad(B)$, choosing $\varepsilon = (\sigma-1)/2$ when $\sigma<2$. The proof of condition \ref{tricky} is more intricate and requires one to follow the steps presented above to arrive at the correct constant, but the argument is identical.

\begin{theorem}\label{maintheoremwithsigma}
Let $(X,d,\mu)$ be a metric measure space with a doubling measure $\mu$. Assume that $\Omega\subset X$ is an open set and that $w$ is a weight on $\Omega$. Let $\sigma>1$. Then the following assertions are equivalent.

\begin{enumerate}[label=\normalfont{(\alph*)}]

\item For every $\eta>0$ there exists an $\varepsilon >0$ such that if $B$ is a ball with $\sigma B\Subset \Omega$ and $F \subset B$ is a measurable set, then 
$\mu(F) \leq \varepsilon \mu(B)$ implies that $w(F) \leq \eta w(\sigma B)$.

\item\label{tricky} There exist constants $\eta, \varepsilon>0$ with $\eta<C_d^{-\left\lfloor \log_2(5 \sigma^2) \right\rfloor -1}$ such that for every ball $B \subset X$ with $\sigma B\Subset \Omega$ and every measurable set $F \subset B$, 
$\mu(F) \leq \varepsilon \mu(B)$ implies that $w(F) \leq \eta w(\sigma B)$.

\item There exist $p>1$ and a constant $C>0$ such that 
$$
\dashint_B w^p \dd \mu \leq C \left(  \dashint_{\sigma B} w \dd \mu \right)^p
$$
for every ball $B$ with $\sigma B\Subset \Omega$.
\item There exist constants $C, \alpha >0$ such that for every ball $B \subset X$ with $\sigma B\Subset \Omega$ and every measurable set $F \subset B$, it holds that 
$$
w(F) \leq C \left( \frac{\mu(F)}{\mu(B)} \right)^\alpha w(\sigma B).
$$

\item There exists a constant $C>0$ for which
$$
\int_B M(w\Chi_B ) \dd \mu \leq C w(\sigma B)
$$
for every ball $B$ with $\sigma B \Subset \Omega$.
\item There exist constants $\alpha, \beta>0$ with $\beta< C_d^{-\left\lfloor \log_2(5 \sigma^2) \right\rfloor -1}$ such that
$$
w \left( B\cap\left\lbrace w\geq  \alpha w_{\sigma B} \right\rbrace \right) \leq \beta w(\sigma B)
$$
for every ball $B$ with $\sigma B\Subset \Omega$.

\item There exists a constant $C>0$ such that for every ball $B$ with $\sigma B\Subset \Omega$, 
$$
\int_B w \log^+\left(\frac{w}{w_{\sigma B}} \right) \dd \mu \leq C w(\sigma B).
$$

\item There exists a nondecreasing function $\phi:(0,\infty)\to(0,\infty)$ with $\phi(0^+)=0$ such that for every ball $B \subset X$ with $\sigma B\Subset \Omega$ and every measurable set $F \subset B$ it holds that
$$
w(F) \leq \phi \left( \frac{\mu(F)}{\mu(B)} \right) w(\sigma B).
$$

\item There exists a constant $C>0$ such that for every ball $B$ with $\kappa B \Subset \Omega$, where $\kappa = \max\lbrace \sigma, 11\rbrace$, and every function $f\in \BMO(\Omega)$ with $\norm{f}_{\BMO(\Omega)} \leq 1$, it holds that
$$
\int_B\abs{ f-f_B } w \dd \mu \leq C w(\sigma B).
$$

\item There exist a constant $C>0$ and a nondecreasing function $\phi: (1,\infty) \to (0,\infty)$ with $\phi(\infty) = \infty$ such that for every ball $B$ with $\sigma B\Subset \Omega$, 
$$
\int_{B\cap\lbrace w>w_{\sigma B} \rbrace} w \, \phi \left(\frac{w}{w_{\sigma B}} \right) \dd \mu \leq C w(\sigma B).
$$
\end{enumerate}
In the case where $w(\sigma B)$ equals zero, the inequalities are trivially satisfied for $B.$
\end{theorem}

\end{document}